\documentclass[11pt, reqno]{amsart}
\usepackage{amsfonts}
\usepackage{amssymb}
\usepackage{latexsym}
\usepackage[final]{hyperref}
\usepackage[]{enumerate}
\usepackage{verbatim}
\usepackage{graphicx}       % Handles figures
\usepackage[utf8]{inputenc} % We want æøå
\usepackage{amsmath}

\usepackage{tikz}
\usepackage{cmap}
\usepackage{ulem}
\usepackage{soul}
\usepackage{comment}
\usepackage{fancybox}
\usepackage{relsize}
\usepackage{pifont}

\newcommand{\norm}[1]{\left\Vert #1 \right\Vert}

\DeclareMathOperator{\supp}{supp}
\DeclareMathOperator{\trace}{tr}

\newtheorem{theorem}{Theorem}
\newtheorem{proposition}{Proposition}
\newtheorem{lemma}{Lemma}

\theoremstyle{definition}

\theoremstyle{remark}

\title
[
	Stochastic BBM equation 
]
{
	A stochastic Benjamin-Bona-Mahony type equation
}

\author[Dinvay]{Evgueni Dinvay}

\email{ Evgueni.Dinvay@inria.fr }
 
\address
{
	Inria Rennes - Bretagne Atlantique
	\\
	Campus universitaire de Beaulieu Avenue du G\'en\'eral Leclerc
	\\
	35042 Rennes Cedex
	\\
	France
}

\date{\today}

 %, 35Q55, 35A01, 35Q20, 35G25, 35D99

\begin{document}

\begin{abstract}
Considered herein is a particular nonlinear dispersive stochastic equation.
It was introduced recently in \cite{Dinvay_Memin2022},
as a model describing surface water waves under location uncertainty.
The corresponding noise term is introduced through a Hamiltonian formulation,
which guarantees the energy conservation of the flow.
Here the initial-value problem is studied.
\end{abstract}

%%%%%%%%%%%%%%%%%%%%%%%%%%%%%%%%%%%%%%%%%%%%%%%%%%%%%%%%%%%%%%%%%%%%%%%%%%%%%
%%%%%%%%%%%%%%%%%%%%%%%%%%%%%%%%%%%%%%%%%%%%%%%%%%%%%%%%%%%%%%%%%%%%%%%%%%%%%
\maketitle
%%%%%%%%%%%%%%%%%%%%%%%%%%%%%%%%%%%%%%%%%%%%%%%%%%%%%%%%%%%%%%%%%%%%%%%%%%%%%
%%%%%%%%%%%%%%%%%%%%%%%%%%%%%%%%%%%%%%%%%%%%%%%%%%%%%%%%%%%%%%%%%%%%%%%%%%%%%
\section{Introduction}
%%%%%%%%%%%%%%%%%%%%%%%%%%%%%%%%%%%%%%%%%%%%%%%%%%%%%%%%%%%%%%%%%%%%%%%%%%%%%
\setcounter{equation}{0}
%%%%%%%%%%%%%%%%%%%%%%%%%%%%%%%%%%%%%%%%%%%%%%%%%%%%%%%%%%%%%%%%%%%%%%%%%%%%%

Consideration is given to the following Stratonovich one-dimensional
BBM-type equation
\begin{equation}
\label{Stratonovich_BBM}
	d u
	=
	- \partial_x K
	\left(
		u + K u^2
	\right)
	d t
	+ \sum_j \gamma_j \partial_x
	\left(
		u + K u^2
	\right)
	\circ d W_j
\end{equation} 
introduced in \cite{Dinvay_Memin2022},
as a model describing surface waves of a fluid layer.
It is supplemented with the initial condition
\(
	u(0) = u_0
	.
\)
Equation \eqref{Stratonovich_BBM} has
a Hamiltonian structure with the energy
\begin{equation}
\label{Hamiltonian}
	\mathcal H (u)
	=
	\int_{\mathbb R}
	\left(
		\frac 12 \big( K^{-1/2}u \big)^2 + \frac 13 u^3
	\right)
	dx
	.
\end{equation}
The Fourier multiplier operator $K$,
defined in the space of tempered distributions $\mathcal S'(\mathbb R)$,
has an even symbol of the form
\begin{equation}
\label{K_definition}
	K(\xi) \simeq ( 1 + \xi^2 )^{ - \sigma_0 }
\end{equation}
with $\sigma_0 > 1/2$.
Expression \eqref{K_definition} means that the symbol $K(\xi)$
is bounded from below and above by $\mbox{RHS\eqref{K_definition}}$
multiplied by some positive constants.
In other words the operator $K$ essentially behaves as the Bessel potential
of order $2 \sigma_0$, see \cite{Grafakos2009}.
The space variable is $x \in \mathbb R$ and the time variable
is $t \geqslant 0$.
The unknown $u$ is a real valued function of these
variables and of the probability variable $\omega \in \Omega$,
representing the free surface elevation
in the fluid layer.
The scalar sequence $\{ \gamma_j \}$ satisfies the restriction
\(
	\sum_j \gamma_j^2 < \infty
	,
\)
and $\{ W_j \}$ is a sequence of independent scalar Brownian motions
on a filtered probability space
\(
	\left(
		\Omega, \mathcal F,
		\{ \mathcal F_t \},
		\mathbb P
	\right)
	.
\)

Model \eqref{Stratonovich_BBM} was introduced in \cite{Dinvay_Memin2022},
where an attempt to extend an elegant Hamiltonian formulation
of \cite{Craig_Groves} to the stochastic setting was made.
We will just briefly comment on the methodology of \cite{Dinvay_Memin2022}.
The white noise is firstly introduced via the stochastic transport theory presented
in \cite{Memin2014},
which is based on splitting of fluid particle motion into smooth and random movements.
Then it is restricted to a particular Stratonovich form in order to respect
the energy conservation.
In particular, it provides us with a model having multiplicative noise
of Hamiltonian structure.
Finally, a long wave approximation results in simplified models as \eqref{Stratonovich_BBM},
for example.

One may notice that after discarding the nonlinear terms in Equation \eqref{Stratonovich_BBM},
the details can be seen in \cite{Dinvay_Memin2022},
the corresponding linearised initial-value problem can be solved exactly
with the help of the fundamental multiplier operator
\begin{equation}
\label{general_Duhamel_matrix}
	\mathcal S(t, t_0)
	=
	\exp
	\left[
		- \partial_x K ( t - t_0 ) + \sum_j \gamma_j \partial_x ( W_j(t) - W_j(t_0) )
	\right]
	,
\end{equation}
where $t_0, t \in \mathbb R$.
Note that it can be factorised as
\(
	\mathcal S(t, t_0) = S(t - t_0) S_W(t, t_0)
	,
\)
where $S(t) = \exp( - \partial_x Kt )$ is a unitary semi-group
and $S_W$ containing all the randomness coming from the Wiener process is unitary as well.
They obviously commute as bounded differential operators.
We recall that $S(t)$ is defined via the Fourier transform
\(
	\mathfrak F \left( S(t) \psi \right)
	=
	\exp( - i\xi K(\xi) t ) \widehat \psi(\xi)
\)
for any $\psi \in \mathcal S'(\mathbb R)$ and
\(
	\widehat \psi = \mathfrak F \psi
	.
\)
Similarly, $S_W(t, t_0)$ is defined by the line
\[
	S_W(t, t_0) \psi
	=
	\mathfrak F^{-1}
	\left(
		\xi \mapsto \exp
		\left(
			i\xi \sum_j \gamma_j ( W_j(t) - W_j(t_0) )
		\right) \widehat \psi(\xi)
	 \right)
	.
\]
It allows us to represent \eqref{Stratonovich_BBM} in the Duhamel form
\begin{equation}
\label{general_Duhamel}
	u(t)
	=
	\mathcal S(t, 0)
	\left(
		u_0
		+
		\int_0^t \mathcal S(0, s) f(u(s)) ds
		+
		\sum_j \gamma_j
		\int_0^t \mathcal S(0, s) g(u(s)) dW_j(s)
	\right)
	,
\end{equation}
where
\begin{equation*}
	f(u)
	=
	- \partial_x K^2 u^2
	+ \sum_j \gamma_j^2
	\partial_x K( u \partial_x K u^2)
\end{equation*}
and
\begin{equation*}
	g(u)
	=
	\partial_x K u^2
	.
\end{equation*}
Existence and uniqueness of solution to
Equation \eqref{general_Duhamel} is under consideration.
It is worth to point out that both $S_W$
and the stochastic integral in \eqref{general_Duhamel} are well defined.
Indeed, appealing to Doobs' inequalities for the submartingale
\(
	\left|
		\sum_{j = n}^{n + m} \gamma_j W_j
	\right|
\)
and the It\^o-Nisio theorem one can show that
\(
	\sum_j \gamma_j W_j
\)
converges uniformly in time almost surely, in probability
and in $L^2$ sense.
If the integrand of the stochastic integral in \eqref{general_Duhamel}
is in some Sobolev space $H^{\sigma}( \mathbb R )$ for each $s$ and a.e. $\omega$,
then we can understand this sum of integrals as an integration
with respect to a $Q$-Wiener process associated with
a Hilbert space $H$ and a non-negative symmetric trace class operator
$Q$ having eigenvalues $\gamma_j^2$
and eigenfunctions $e_j$ forming an orthonormal basis in $H$.
Then the corresponding integrand is the unbounded linear operator
between $H$ and $H^{\sigma}( \mathbb R )$ that maps all $e_j$
to the same element of $H^{\sigma}( \mathbb R )$, namely, to
\(
	\mathcal S(0, s) g(u(s))
	.
\)
In particular, it explains why we need
the summability condition
\(
	\sum_j \gamma_j^2 < \infty
	.
\)

Before we formulate the main result it is left
to introduce a notation as follows. 
By
\(
	C(0, T; H^{\sigma}(\mathbb R))
\)
we will notate the space of continuous functions on $[0, T]$
having values in $H^{\sigma}(\mathbb R)$
with the usual supremum norm.

\begin{theorem}
\label{main_theorem}
%	Let
%	\(
%		\sigma \geqslant \sigma_0 > 1/2
%		.
%	\)
	Let $\sigma_0 > 1/2$ and $\sigma \geqslant \max \{ \sigma_0, 1 \}$.
	Then for any
	$\mathcal F_0$-measurable
	\(
		u_0 \in
		L^2(\Omega; H^{\sigma}( \mathbb R ))
		\cap
		L^{\infty}(\Omega; H^{\sigma_0}( \mathbb R ))
	\)
	with sufficiently small $L^{\infty}H^{\sigma_0}$-norm
	and any $T_0 > 0$
	Equation \eqref{general_Duhamel} has
	%there is
	a unique adapted solution
	\(
		u \in L^2( \Omega; C(0, T_0; H^{\sigma}(\mathbb R)) )
		\cap L^{\infty}( \Omega; C(0, T_0; H^{\sigma_0}(\mathbb R)) )
		.
	\)
	%solving Equation \eqref{general_Duhamel}.
	Moreover,
	\(
		\mathcal H(u(t)) = \mathcal H(u_0)
	\)
	for each $t \in [0, T_0]$ almost surely on $\Omega$.
\end{theorem}

The conservation of energy \eqref{Hamiltonian} plays a crucial role
in the proof.
So it will be a bit more convenient to
regard the energy norm defined by
\[
	\norm{u}_{ \mathcal H }^2
	=
	\frac 12 \int_{\mathbb R}
	\big( K^{-1/2}u \big)^2
	dx
\]
instead of the spatial $H^{\sigma_0}$-norm.
They are obviously equivalent.

The proof is essentially based on the contraction mapping principle.
We do not exploit much smoothing properties of the group $\mathcal S(t, t_0)$,
as for example is done in \cite{Bouard_Debussche1999}
for analysis of a stochastic nonlinear Schr\"odinger equation.
It is enough to know that the absolute value of its symbol equals one,
and that $S(t)$ is a unitary semigroup.
However, in order to appeal to the fixed point theorem
we have to truncate both deterministic $f$ and random $g$ nonlinearities.
There are a couple of technical difficulties related to
implementation of the energy conservation in our case.
Firstly, for the truncated equation we can claim $\mathcal H$-conservation
only until a particular stopping time.
Secondly, one can control $\norm{u}_{ \mathcal H }$ with $\mathcal H(u)$
only provided $\norm{u}_{ \mathcal H }$ is small.
These additional difficulties
make us repeat the arguments of the last section in the paper iteratively
in order to construct solution on the whole  time interval $[0, T_0]$.
%This results in additional technical difficulties while   
%Different dispersive stochastic ...

As a final remark we point out that the noise in Equation \eqref{Stratonovich_BBM}
can be gathered in one dimensional
\(
	\partial_x
	\left(
		u + K u^2
	\right)
	\circ d B
\)
with the scalar Brownian motion
\(
	B = \sum_j \gamma_j W_j
	.
\)
However, this does not affect the proof below anyhow,
so we continue to stick to the original formulation \eqref{Stratonovich_BBM}.
In future works we are planning to extend it to $\gamma_j$
being either Fourier multipliers or space-dependent coefficients.

\section{Truncation}
%%%%%%%%%%%%%%%%%%%%%%%%%%%%%%%%%%%%%%%%%%%%%%%%%%%%%%%%%%%%%%%%%%%%%%%%%%%%%
\setcounter{equation}{0}
%%%%%%%%%%%%%%%%%%%%%%%%%%%%%%%%%%%%%%%%%%%%%%%%%%%%%%%%%%%%%%%%%%%%%%%%%%%%%

The Sobolev space $H^{\sigma}(\mathbb R)$
consists of tempered distributions $u$
having the finite square norm
\(
	\norm{u}_{H^{\sigma}}^2
	=
	\int \left| \widehat u (\xi) \right| ^2
	\left( 1 + \xi^2 \right) ^{\sigma} d\xi
	<
	\infty
	.
\)
Let
\(
	\theta \in C_0^{\infty}(\mathbb R)
\)
with
\(
	\supp \theta \in [-2, 2]
\)
being such that
\(
	\theta(x) = 1
\)
for
\(
	x \in [-1, 1]
\)
and
\(
	0 \leqslant \theta(x) \leqslant 1
\)
for
\(
	x \in \mathbb R
	.
\)
For $R > 0$ we introduce the cut off
$\theta_R(x) = \theta(x/R)$ and
\[
	f_R(u)
	=
	\theta_R( \norm{u}_{H^{\sigma}} ) f(u)
	, \quad
	g_R(u)
	=
	\theta_R( \norm{u}_{H^{\sigma}} ) g(u)
\]
that we substitute in \eqref{general_Duhamel}
instead of $f(u)$, $g(u)$, respectively.
The new
$R$-regularisation of \eqref{general_Duhamel}
reads as
\begin{equation}
\label{R_Duhamel}
	u(t)
	=
	\mathcal S(t, t_0)
	\left(
		u(t_0)
		+
		\int_{t_0}^t \mathcal S(t_0, s) f_R(u(s)) ds
		+
		\sum_j \gamma_j
		\int_{t_0}^t \mathcal S(t_0, s) g_R(u(s)) dW_j(s)
	\right)
	.
\end{equation}
In this section without loss of generality
we can set $t_0 = 0$ and $u(t_0) = u_0$.
We will vary time moments $t_0$ below in the next section.
Equation \eqref{R_Duhamel} can be solved with a help
of the contraction mapping principle
in
\(
	L^2( \Omega; C(0, T; H^{\sigma}(\mathbb R)) )
	.
\)

\begin{proposition}
\label{R_existence_proposition}
	Let $\sigma > 1/2$,
	\(
		u_0 \in L^2( \Omega; H^{\sigma}(\mathbb R) )
	\)
	be $\mathcal F_0$-measurable and $T_0 > 0$.
	Then \eqref{R_Duhamel} has a unique adapted solution
	\(
		u \in L^2( \Omega; C(0, T_0; H^{\sigma}(\mathbb R)) )
		.
	\)
	Moreover, it depends continuously on the initial data $u_0$.
\end{proposition}
\begin{proof}
We set
\(
	\mathcal Tu(t) = \mbox{RHS\eqref{R_Duhamel}}
	.
\)
We will show that $\mathcal T$ is a contraction mapping in
\(
	X_T = L^2( \Omega; C(0, T; H^{\sigma}(\mathbb R)) )
	,
\)
provided $T > 0$ is sufficiently small, depending only on $R$.
Let $u_1, u_2$ be two adapted processes in $X_T$.
Firstly, one can notice that
\[
	\norm
	{
		f_R( u_1 ) - f_R( u_2 )
	}
	_{H^{\sigma}}
	\leqslant
	C \left( 1 + R \right)^2
	\norm
	{
		u_1 - u_2
	}
	_{H^{\sigma}}
	,
\]
\[
	\norm
	{
		g_R( u_1 ) - g_R( u_2 )
	}
	_{H^{\sigma}}
	\leqslant
	C R
	\norm
	{
		u_1 - u_2
	}
	_{H^{\sigma}}
	.
\]
Indeed, $H^{\sigma}(\mathbb R)$ poses an algebraic property
for $\sigma > 1/2$
and $\partial_x K$ is bounded in $H^{\sigma}(\mathbb R)$.
Then assuming
\(
	\norm{u_1}_{H^{\sigma}} \geqslant \norm{u_2}_{H^{\sigma}}
\)
without loss of generality one deduces
\begin{multline*}
	\norm
	{
		g_R( u_1 ) - g_R( u_2 )
	}
	_{H^{\sigma}}
	\leqslant
	C
	\norm
	{
		\theta_R( \norm{u_1}_{H^{\sigma}} ) u_1^2
		-
		\theta_R( \norm{u_2}_{H^{\sigma}} ) u_2^2
	}
	_{H^{\sigma}}
	\\
	\leqslant
	C \theta_R( \norm{u_1}_{H^{\sigma}} )
	\norm
	{
		u_1^2 - u_2^2
	}
	_{H^{\sigma}}
	+
	\left|
		\theta_R( \norm{u_1}_{H^{\sigma}} )
		-
		\theta_R( \norm{u_2}_{H^{\sigma}} )
	\right|
	\norm
	{
		u_2^2
	}
	_{H^{\sigma}}
	\leqslant
	C R
	\norm
	{
		u_1 - u_2
	}
	_{H^{\sigma}}
	,
\end{multline*}
where we have used the estimate
\(
	\left|
		\theta_R( \norm{u_1}_{H^{\sigma}} )
		-
		\theta_R( \norm{u_2}_{H^{\sigma}} )
	\right|
	\leqslant
	\norm{\theta'}_{L^{\infty}} R^{-1}
	\norm{u_1 - u_2}_{H^{\sigma}}
\)
following obviously from the mean value theorem.
The difference between $f_R( u_1 )$ and $f_R( u_2 )$
can be obtained in the same way.
Thus
\begin{multline*}
	\norm
	{
		\mathcal Tu_1(t) - \mathcal Tu_2(t)
	}
	_{H^{\sigma}}
	\leqslant
	\norm
	{
		\int_0^t \mathcal S(0, s) ( f_R( u_1(s) ) - f_R( u_2(s) ) ) ds
	}
	_{H^{\sigma}}
	\\
	+
	\norm
	{
		\sum_j \gamma_j
		\int_0^t \mathcal S(0, s) ( g_R( u_1(s) ) - g_R( u_2(s) ) ) dW_j(s)
	}
	_{H^{\sigma}}
	= I + II
	.
\end{multline*}
The first integral is estimated straightforwardly as
\[
	I
	\leqslant
	\int_0^T
	\norm
	{
		f_R( u_1(s) ) - f_R( u_2(s) )
	}
	_{H^{\sigma}}
	ds
	\leqslant
	C ( 1 + R )^2 T
	\norm
	{
		u_1 - u_2
	}
	_{C( 0, T; H^{\sigma} )}
	.
\]
The second one is estimated with the use of the Burkholder inequality
\cite{Gawarecki_Mandrekar} as
\[
	\mathbb E \sup_{ 0 \leqslant t \leqslant T } II^2
	\leqslant
	C \mathbb E \int_0^T
	\norm
	{
		g_R( u_1(s) ) - g_R( u_2(s) )
	}
	_{H^{\sigma}}^2
	ds
	\leqslant
	C R^2 T \mathbb E
	\norm
	{
		u_1 - u_2
	}
	_{C( 0, T; H^{\sigma} )}^2
	.
\]
It is clear that time-continuity of
\(
	\mathcal Tu_1
	,
	\mathcal Tu_2
\)
follows from the factorisation
\(
	\mathcal S = S S_W
\)
and the estimate
\(
	\norm
	{
		S_W g_R( u )
	}
	_{H^{\sigma}}
	\leqslant
	CR^2
	,
\)
so we have a stochastic convolution as in
\cite[Lemma 3.3]{Gawarecki_Mandrekar}.
Thus
\[
	\norm
	{
		\mathcal Tu_1 - \mathcal Tu_2
	}
	_{X_T}
	\leqslant
	C \left( ( 1 + R )^2 T + R \sqrt{T} \right)
	\norm
	{
		u_1 - u_2
	}
	_{X_T}
	,
\]
and so there exists a small $T$ depending only on $R$
such that $\mathcal T$ has a unique fixed point in $X_T$.
Moreover, this estimate also gives us continuous dependence
of solution in $X_T$ on the initial data
\(
	u_0 \in L^2( \Omega; H^{\sigma}(\mathbb R) )
	,
\)
obviously.
Clearly, the solution can be extended
to the whole interval $[0, T_0]$.
\end{proof}

The regularisation affects the energy conservation.
Indeed, in the It\^o differential form Equation \eqref{R_Duhamel}
reads
\begin{equation}
\label{R_Ito_BBM}
	d u
	=
	\left(
		- \partial_x K u
		+ \frac 12 \sum_j \gamma_j^2 \partial_x^2 u
		+ f_R(u)
		+ \sum_j \gamma_j^2 \partial_x g_R(u)
	\right)
	d t
	+ \sum_j \gamma_j
	\left(
		\partial_x u + g_R(u)
	\right)
	d W_j
	,
\end{equation} 
and so applying the It\^o formula to the energy functional
$\mathcal H(u(t))$ defined by \eqref{Hamiltonian}
with the use of \eqref{R_Ito_BBM},
one can easily obtain
\begin{equation}
\label{R_conservation}
	d \mathcal H(u)
	=
	\left(
	\left(
		\theta_R - 1
	\right)
	\int u^2 \partial_x K u dx
	+
	\theta_R
	\left(
		\theta_R - 1
	\right)
	\sum_j \gamma_j^2
	\int
	\left(
		\frac 12 g(u) K^{-1} g(u)
		+
		u g^2(u)
	\right)
	dx
	\right)
	dt
	.
\end{equation}
Indeed, assuming $\sigma \geqslant \sigma_0 + 2$ at first,
we notice that the solution $u$ given by Proposition \ref{R_existence_proposition}
solves Equation \eqref{R_Ito_BBM}.
Let us introduce the following notations
\[
	\Psi(t)dt + \Phi(t)dW
	=
	\Psi(t)dt
	+ \sum_j \gamma_j
	\Phi(t) e_j d W_j
	=
	\mbox{RHS\eqref{R_Ito_BBM}}
	.
\]
Then It\^o's formula reads
\begin{multline*}
	\mathcal H(u(t))
	=
	\mathcal H(u_0)
	\\
	+
	\int_0^t \partial_u \mathcal H(u(s)) \Psi(s) ds
	+
	\int_0^t \partial_u \mathcal H(u(s)) \Phi(s) dW(s)
	+
	\frac 12 \int_0^t \trace \partial_u^2 \mathcal H(u(s)) ( \Phi(s), \Phi(s) ) ds
	,
\end{multline*}
where the Fr\'echet derivatives are defined by
\[
	\partial_u \mathcal H(u) \phi
	=
	\int_{\mathbb R}
	\left(
		K^{-1/2}u K^{-1/2}\phi + u^2 \phi
	\right)
	dx
	,
\]
\[
	\partial_u^2 \mathcal H(u) ( \phi, \psi)
	=
	\int_{\mathbb R}
	\left(
		K^{-1/2}\phi K^{-1/2}\psi + 2u \phi \psi
	\right)
	dx
\]
at every
\(
	\phi, \psi \in H^{\sigma_0}(\mathbb R)
	.
\)
Substituting these expressions together with
the definitions of $\Phi$ and $\Psi$ into the It\^o's formula
one obtains \eqref{R_conservation}.
Let us, for example, calculate the stochastic integral
\[
	\int_0^t \partial_u \mathcal H(u(s)) \Phi(s) dW(s)
	=
	\sum_j \gamma_j
	\int_0^t \int_{\mathbb R}
	\left(
		K^{-1/2}u K^{-1/2} + u^2
	\right)
	\left(
		\partial_x u
		+
%		g_R(u)
		\theta_R( \norm{u}_{H^{\sigma}} ) \partial_x K u^2
	\right)
	dx
	d W_j
\]
that equals zero as one can see integrating by parts
in the space integral.
Similarly, one calculates the other two integrals
in the It\^o formula.
Thus we have proved \eqref{R_conservation} for
$\sigma \geqslant \sigma_0 + 2$.
In order to lower the bound for $\sigma$,
one would like to argue here by approximation of
initial value $u_0$ via smooth functions
and appeal to the continuous dependence on $u_0$, however,
there is a problem here,
since $\theta_R$ in \eqref{R_conservation}
contains the dependence on $\sigma$.
So even for a smooth initial data the corresponding solution
lies a priori only in $H^{\sigma}$.
This difficulty is overcome in the next statement,
where we argue similar to \cite{Bouard_Debussche2003}.

\begin{proposition}
\label{R_conservation_proposition}
	Let $\sigma_0 > 1/2$ and $\sigma \geqslant \max \{ \sigma_0, 1 \}$.
	Then \eqref{R_conservation} holds almost surely
	for $u$ satisfying Equation \eqref{R_Duhamel}
	given by Proposition \ref{R_existence_proposition}.
\end{proposition}

\begin{proof}

The main idea is to cut off high frequencies
of the differential operator $\partial_x$ in \eqref{R_Ito_BBM} as follows.
Let $P_{\lambda}$ be a Fourier multiplier with the symbol $\theta_{\lambda}$,
$\lambda > 0$.
It is defined by the expression
\(
	\mathfrak F ( P_{\lambda} \psi )
	=
	\theta_{\lambda} \widehat \psi
	.
\)
Now we consider instead of \eqref{R_Ito_BBM} the following regularisation
\begin{equation}
\label{lambda_R_Ito_BBM}
	d u
	=
	\left(
		- \partial_x K u
		+ \frac 12 \sum_j \gamma_j^2 \partial_x^2 P_{\lambda}^2 u
		+ f_R(u)
		+ \sum_j \gamma_j^2 \partial_x P_{\lambda} g_R(u)
	\right)
	d t
	+ \sum_j \gamma_j
	\left(
		\partial_x P_{\lambda} u + g_R(u)
	\right)
	d W_j
\end{equation} 
that has a strong solution.
Indeed, it contains only bounded operators and
the corresponding mild equation has exactly the same form
as Equation \eqref{R_Duhamel} with
\(
	\mathcal S^{\lambda} = S S_W^{\lambda}
\)
now instead of $\mathcal S$, where
\[
	S_W^{\lambda}
	=
	\exp
	\left[
		\sum_j \gamma_j \partial_x P_{\lambda} ( W_j(t) - W_j(t_0) )
	\right]
	.
\]
So we can actually apply Proposition \ref{R_existence_proposition}
to obtain $u = u_{\lambda}$ solving \eqref{lambda_R_Ito_BBM}.
Let $u = u_{\infty}$ stay for the solution of
the original equation \eqref{R_Duhamel}.
Firstly, we will check that
\(
	u_{\lambda} \to u_{\infty}
\)
in
\(
	L^2( \Omega; L^2(0, T_0; H^{\sigma}(\mathbb R)) )
\)
for any $\sigma > 1/2$ as $\lambda \to \infty$.

Let
\(
	0 \leqslant t \leqslant T \leqslant T_0
	,
\)
where a positive small enough time moment $T$
is to be chosen below.
Then
\begin{multline*}
	\norm
	{
		u_{\lambda}(t)
		-
		u_{\infty}(t)
	}
	_{H^{\sigma}}
	=
	\norm
	{
		\mathcal T^{\lambda}u_{\lambda}(t)
		-
		\mathcal T^{\infty}u_{\infty}(t)
	}
	_{H^{\sigma}}
	\leqslant
	\norm
	{
		\left(
			\mathcal S^{\lambda}(t, 0)
			-
			\mathcal S^{\infty}(t, 0)
		\right)
		u_0
	}
	_{H^{\sigma}}
	\\
	+
	\norm
	{
		\int_0^t
		\left(
			\mathcal S^{\lambda}(t, s)
			-
			\mathcal S^{\infty}(t, s)
		\right)
		f_R( u_{\infty}(s) )  ds
	}
	_{H^{\sigma}}
	+
	\norm
	{
		\int_0^t \mathcal S^{\lambda}(t, s)
		( f_R( u_{\lambda}(s) ) - f_R( u_{\infty}(s) ) ) ds
	}
	_{H^{\sigma}}
	\\
	+
	\norm
	{
		\left(
			\mathcal S^{\lambda}(t, 0)
			-
			\mathcal S^{\infty}(t, 0)
		\right)
		\sum_j \gamma_j
		\int_0^t \mathcal S^{\infty}(0, s) g_R( u_{\infty}(s) ) dW_j(s)
	}
	_{H^{\sigma}}
	\\
	+
	\norm
	{
		\sum_j \gamma_j \int_0^t
		\left(
			\mathcal S^{\lambda}(0, s)
			-
			\mathcal S^{\infty}(0, s)
		\right)
		g_R( u_{\infty}(s) ) dW_j(s)
	}
	_{H^{\sigma}}
	\\
	+
	\norm
	{
		\sum_j \gamma_j
		\int_0^t \mathcal S^{\lambda}(0, s)
		( g_R( u_{\lambda}(s) ) - g_R( u_{\infty}(s) ) ) dW_j(s)
	}
	_{H^{\sigma}}
	= I_1 + \ldots + I_6
	.
\end{multline*}
The terms $I_3$ and $I_6$ are estimated exactly as
the analogous integrals $I$ and $II$
in the proof of Proposition \ref{R_existence_proposition},
namely,
\[
	I_3
	\leqslant
	C ( 1 + R )^2 \sqrt T
	\norm
	{
		u_{\lambda} - u_{\infty}
	}
	_{L^2( 0, T; H^{\sigma} )}
\]
and
\[
	\mathbb E \sup_{ 0 \leqslant t \leqslant T } I_6^2
	\leqslant
	C \mathbb E \int_0^T
	\norm
	{
		g_R( u_{\lambda}(s) ) - g_R( u_{\infty}(s) )
	}
	_{H^{\sigma}}^2
	ds
	\leqslant
	C R^2 \mathbb E
	\norm
	{
		u_{\lambda} - u_{\infty}
	}
	_{L^2( 0, T; H^{\sigma} )}^2
	.
\]
Thus
\[
	\mathbb E \int_0^T
	\left(
		I_3^2 + I_6^2
	\right)
	dt
	\leqslant
	C \left( ( 1 + R )^4 T^2 + R^2 T \right)
	\mathbb E
	\norm
	{
		u_{\lambda} - u_{\infty}
	}
	_{L^2( 0, T; H^{\sigma} )}^2
	,
\]
and so there exists a small $T > 0$ depending only on $R$
such that
\[
	\mathbb E
	\norm
	{
		u_{\lambda} - u_{\infty}
	}
	_{L^2( 0, T; H^{\sigma} )}^2
	\leqslant
	C \mathbb E \int_0^T
	\left(
		I_1^2 + I_2^2 + I_4^2 + I_5^2
	\right)
	dt
	.
\]
One needs to show that the right hand side of this expression
tends to zero when $\lambda \to \infty$.
All these four integrals are treated similarly.
Indeed, let us regard more closely the first one
\[
	I_1^2
	=
	\int
	\left|
		\exp
		\left(
			i\xi \theta_{\lambda}(\xi) \sum_j \gamma_j W_j(t)
		\right)
		-
		\exp
		\left(
			i\xi \sum_j \gamma_j W_j(t)
		\right)
	\right| ^2
	\left|
		\widehat{u_0} (\xi)
	\right| ^2
	\left( 1 + \xi^2 \right)^{\sigma} d\xi
\]
that obviously tends to zero as $\lambda \to \infty$
for a.e. $\omega$ and any $t$.
Hence
\(
	\mathbb E \int_0^T I_1^2 dt \to 0
\)
by the dominated convergence theorem, sine
\(
	I_1 \leqslant 2 \norm{u_0}_{H^{\sigma}}
	.
\)
The integral of $I_4^2$ is estimated exactly in the same manner
with the stochastic integral of
\(
	\mathcal S^{\infty} g_R( u_{\infty} )
\)
standing in place of $u_0$.
The second integral
\[
	\mathbb E \int_0^T I_2^2 dt
	\leqslant
	T \mathbb E \int_0^T \int_0^T
	\norm
	{
		\left(
			\mathcal S^{\lambda}(t, s)
			-
			\mathcal S^{\infty}(t, s)
		\right)
		f_R( u_{\infty}(s) )
	}
	_{H^{\sigma}}^2
	ds dt
	\to 0
\]
by the dominated convergence theorem,
since
\(
	\norm
	{
		\ldots
	}
	_{H^{\sigma}}^2
	\leqslant CR^2 (1 + R)^4
	.
\)
Finally, the last integral
\[
	\mathbb E \int_0^T I_5^2 dt
	\leqslant
	T \mathbb E \sup _{ t \in [0, T] } I_5^2
	\leqslant
	C T \mathbb E \int_0^T
	\norm
	{
		\left(
			\mathcal S^{\lambda}(0, s)
			-
			\mathcal S^{\infty}(0, s)
		\right)
		g_R( u_{\infty}(s) )
	}
	_{H^{\sigma}}^2
	ds
	\to 0
\]
by the Burkholder inequality and the dominated convergence theorem,
since
\(
	\norm
	{
		\ldots
	}
	_{H^{\sigma}}^2
	\leqslant CR^4
	.
\)

Repeating this argument iteratively
on subintervals of $[0, T_0]$ of the size $T$ one obtains that
\(
	u_{\lambda} \to u_{\infty}
\)
in
\(
	L^2( \Omega \times [0, T_0] ; H^{\sigma}(\mathbb R) )
	.
\)

Let us calculate each term in the It\^o formula for $u = u_{\lambda}$.
As we shall see the corresponding stochastic integral is not zero,
and moreover, it is difficult to pass to the limit $\lambda \to \infty$
treating the stochastic part.
So instead of $\mathcal H$ we consider at first a sequence
$\mathcal H_n$, $n \in \mathbb N$, with the cubic term being cut off
in the following way
\[
	\mathcal H_n(u)
	=
	\norm u_{\mathcal H}^2
	+
	\frac 13
	\theta_n
	\left(
		\norm u_{\mathcal H}^2
	\right)
	\int u^3 dx
\]
that clearly tends to $\mathcal H(u)$ almost surely at any
fixed time moment.
The corresponding
Fr\'echet derivatives are defined by
\[
	\partial_u \mathcal H_n(u) \phi
	=
	\int_{\mathbb R}
	\left[
		\left(
			1
			+
			\frac 13
			\theta_n'
			\left(
				\norm u_{\mathcal H}^2
			\right)
			\int u^3 dy
		\right)
		K^{-1/2}u K^{-1/2}\phi
		+
		\theta_n
		\left(
			\norm u_{\mathcal H}^2
		\right)
		u^2 \phi
	\right]
	dx
	,
\]
\begin{multline*}
	\partial_u^2 \mathcal H_n(u) ( \phi, \psi)
	=
	\int_{\mathbb R}
	\left[
		\left(
			1
			+
			\frac 13
			\theta_n'
			\left(
				\norm u_{\mathcal H}^2
			\right)
			\int u^3 dx
		\right)
		K^{-1/2} \phi K^{-1/2}\psi
		+
		2 \theta_n
		\left(
			\norm u_{\mathcal H}^2
		\right)
		u \phi \psi
	\right]
	dx
	\\
	+
	\theta_n'
	\left(
		\norm u_{\mathcal H}^2
	\right)
	\int u^2 \phi dx
	\int K^{-1/2} u K^{-1/2}\psi dy
	\\
	+
	\frac 13
	\theta_n''
	\left(
		\norm u_{\mathcal H}^2
	\right)
	\int u^3 dx
	\int K^{-1/2} u K^{-1/2}\phi dy
	\int K^{-1/2} u K^{-1/2}\psi dz
\end{multline*}
at every
\(
	\phi, \psi \in H^{\sigma_0}(\mathbb R)
	.
\)
Substituting it to the stochastic integral
one obtains the following expression
that can be simplified by integration by parts
\begin{multline*}
	\int_0^t \partial_u \mathcal H_n(u(s)) \Phi(s) dW(s)
	\\
	=
	\sum_j \gamma_j
	\int_0^t \int_{\mathbb R}
	\left[
		\left(
			1
			+
			\frac 13
			\theta_n'
			\left(
				\norm u_{\mathcal H}^2
			\right)
			\int u^3 dy
		\right)
		K^{-1/2}u K^{-1/2}
		+
		\theta_n
		\left(
			\norm u_{\mathcal H}^2
		\right)
		u^2
	\right]
	\\
	\left(
		\partial_x P_{\lambda} u
		+
		\theta_R( \norm{u}_{H^{\sigma}} ) \partial_x K u^2
	\right)
	dx
	d W_j
	=
	\sum_j \gamma_j
	\int_0^t
	\theta_n
	\left(
		\norm u_{\mathcal H}^2
	\right)
	\int_{\mathbb R}
	u^2
	\partial_x P_{\lambda} u
	dx
	d W_j
	,
\end{multline*}
where $u = u_{\lambda}$.
We will show that this integral tends to zero as $\lambda \to \infty$.
That is exactly the place where we need the cut off $\theta_n$.
Applying some algebraic manipulations to the space integral
and the Burkholder inequality to the stochastic integral,
one deduces the estimate
\begin{multline*}
	\mathbb E \sup _{0 \leqslant t \leqslant T_0}
	\left|
		\int_0^t \partial_u \mathcal H_n(u(s)) \Phi(s) dW(s)
	\right| ^2
	\leqslant
	C \mathbb E \int_0^{T_0}
	\theta_n^2
	\left(
		\norm { u_{\lambda}(t) }_{\mathcal H}^2
	\right)
	\left(
		\int_{\mathbb R}
		u_{\lambda}^2(t)
		\partial_x ( P_{\lambda} - 1 ) u_{\lambda}(t)
		dx
	\right) ^2
	dt
	\\
	\leqslant
	C \mathbb E \int_0^{T_0}
	\theta_n^2
	\left(
		\norm { u_{\lambda}(t) }_{\mathcal H}^2
	\right)
	\norm { u_{\lambda}(t) }_{\mathcal H}^4
	\left(
		\norm{
			( P_{\lambda} - 1 ) u_{\infty}(t)
		} _{H^{1/2}}^2
		+
		\norm{
			( P_{\lambda} - 1 ) ( u_{\lambda}(t) - u_{\infty}(t) )
		} _{H^{1/2}}^2
	\right)
	dt
	\\
	\leqslant
	C n^4 \mathbb E \int_0^{T_0}
	\left(
		\norm{
			( P_{\lambda} - 1 ) u_{\infty}(t)
		} _{H^{1/2}}^2
		+
		\norm{
			( u_{\lambda}(t) - u_{\infty}(t) )
		} _{H^{1/2}}^2
	\right)
	dt
	\to 0
\end{multline*}
as $\lambda \to 0$ for each fixed $n \in \mathbb N$.
Note that the use of the functional $\mathcal H_n$
instead of $\mathcal H$ is important here.
Similarly, we calculate the rest two terms in the It\^o formula
\begin{multline*}
	\partial_u \mathcal H_n(u) \Phi
	+
	\frac 12 \trace \partial_u^2 \mathcal H(u) ( \Phi, \Phi )
	=
	\left(
		\theta_R - \theta_n
	\right)
	\int u^2 \partial_x K u dx
	+
	\theta_n \theta_R
	\left(
		\theta_R - 1
	\right)
	\sum_j \gamma_j^2
	\int u g^2(u) dx
	\\
	+
	\frac { \theta_R (\theta_R - 1) }2
	\sum_j \gamma_j^2
	\int g(u) K^{-1} g(u) dx
	+
	\frac {\theta_n}2
	\sum_j \gamma_j^2 \int
	\left(
		u^2 \partial_x^2 P_{\lambda}^2 u
		+
		2u ( \partial_x P_{\lambda} u )^2
	\right)
	dx
	\\
	+
	\theta_n \theta_R
	\sum_j \gamma_j^2
	\left(
		2 \int u( \partial_x P_{\lambda} u ) g(u) dx
		-
		\int g(u) P_{\lambda} K^{-1} g(u) dx
	\right)
	\\
	+
	\frac 13 \theta_R \theta_n' \int u^3 dy
	\left(
		\frac {\theta_R - 1}2 \sum_j \gamma_j^2
		\int g(u) K^{-1} g(u) dx
		-
		\int u g(u) dx
	\right)
	= J_1 + \ldots + J_6
	,
\end{multline*}
where as above $u = u_{\lambda}$.
One can prove that for a.e. $\omega \in \Omega$ and $t \in [0, T_0]$
the first three terms
\(
	J_1 + J_2 + J_3
\)
tend to the integrand of the right hand side of Expression \eqref{R_conservation}
in the subsequent limits, firstly, as $\lambda \to \infty$ and then as $n \to \infty$. 
Both $J_4$ and $J_5$ tend to zero as $\lambda \to \infty$.
Meanwhile the last term $J_6$ stays bounded by $C/n$,
and so
\(
	\lim _{n \to \infty} \lim _{\lambda \to \infty} J_6 = 0
	.
\)
Let us show, for example, that $J_4 \to 0$
which is the most troublesome term in the sum,
since here is the only place in the paper where
we make use of the fact $\sigma \geqslant 1$.
The rest are treated similarly without this additional restriction.
Indeed,
\[
	J_4
	\leqslant
	C
	\left|
		\int
		\left(
			u \partial_x P_{\lambda} u
			-
			P_{\lambda} ( u \partial_x u )
		\right)
		( P_{\lambda} - 1 ) \partial_x u
		dx
	\right|
	\leqslant
	C
	\norm{ u_{\lambda} } _{H^1}^2
	\left(
		\norm{ ( P_{\lambda} - 1 ) u_{\infty} } _{H^1}
		+
		\norm{ u_{\lambda} - u_{\infty} } _{H^1}
	\right)
\]
that obviously tends to zero as $\lambda \to \infty$.
This concludes the proof.

\begin{comment}

The stochastic integral
%
\begin{multline*}
	\int_0^t \partial_u \mathcal H(u(s)) \Phi(s) dW(s)
	\\
	=
	\sum_j \gamma_j
	\int_0^t \int_{\mathbb R}
	\left(
		K^{-1/2}u K^{-1/2} + u^2
	\right)
	\left(
		\partial_x P_{\lambda} u
		+
		\theta_R( \norm{u}_{H^{\sigma}} ) \partial_x K u^2
	\right)
	dx
	d W_j
	\\
	=
	\sum_j \gamma_j
	\int_0^t \int_{\mathbb R}
	\left(
		u^2
		\partial_x P_{\lambda} u
	\right)
	dx
	d W_j
\end{multline*}
%

\end{comment}

\end{proof}

At this stage one cannot claim the energy conservation yet,
so we will prove a weaker result that will be sharpened later.
Note that there exists $C_{\mathcal H} > 0$ such that
\begin{equation}
\label{energy_norm_control}
	\norm{ u }_{ \mathcal H }^2
	( 1 - C_{\mathcal H} \norm{ u }_{ \mathcal H } )
	\leqslant
	\mathcal H(u)
	\leqslant
	\norm{ u }_{ \mathcal H }^2
	( 1 + C_{\mathcal H} \norm{ u }_{ \mathcal H } )
	,
\end{equation}
following from the well-known embedding
\(
	H^{\sigma_0} (\mathbb R)
	\hookrightarrow
	L^{\infty} (\mathbb R)
	,
\)
recall that $\sigma_0 > 1/2$.

\begin{lemma}
\label{energy_estimate_lemma}
	There exists a constant $T_1 > 0$ independent of $\omega$
	such that if $u$ solving Equation \eqref{R_Duhamel} has
	\(
		\norm{ u }_{ \mathcal H }
		\leqslant
		\frac 1{ 2 C_{\mathcal H} }
	\)
	on some interval $[0, \tau]$ then
	\(
		\mathcal H(u) \leqslant 2 \mathcal H(u(0))
	\)
	on
	\(
		[ 0, T_1 \wedge \tau]
		.
	\)
\end{lemma}

\begin{proof}

At first one can notice that as long as $\norm{ u }_{ \mathcal H }$
stays bounded by $(2 C_{\mathcal H})^{-1}$,
we have
\[
	\frac 12 \norm{ u }_{ \mathcal H }^2
	\leqslant
	\mathcal H(u)
	\leqslant
	\frac 32 \norm{ u }_{ \mathcal H }^2
	.
\]
Moreover, one can as well easily deduce from \eqref{R_conservation}
the following bound
\[
	\mathcal H(u(t))
	\leqslant
	\mathcal H(u(0)) + C \int_0^t \mathcal H(u(s)) ds
	,
\]
and so  the proof is concluded by Gr\"onwall's lemma.

\end{proof}

\begin{comment}

\begin{proof}[\textcolor{red}{SIMPLIFY THE PROOF}]

At first one can notice that as long as $\norm{ u }_{ \mathcal H }$
stays bounded by $(2 C_{\mathcal H})^{-1}$,
we have
%
\[
	\frac 12 \norm{ u }_{ \mathcal H }^2
	\leqslant
	\mathcal H(u)
	\leqslant
	\frac 32 \norm{ u }_{ \mathcal H }^2
	.
\]
%
Denoting by $y(t) = \mathcal H (u(t))$ and by $y_0 = y(0) = \mathcal H (u(0))$
one can easily deduce from \eqref{R_conservation} that
%
\[
	y(t) \leqslant y_0 + C \int_0^t y(1 + y)
	.
\]
%
We will find $T(y_0)$ being non-increasing function of $y_0$
such that for any $t \in [0, T]$
it holds that
\[
	y(t) \leqslant 2y_0
	.
\]
%
%\begin{proof}
%
%
Indeed, introduce
\[
	Y(t) = y_0 + C \int_0^t y(1 + y)
\]
and notice that
\[
	\left(
		\log \frac Y{1 +Y}
	\right) '
	=
	\frac {Y'}{ Y (1 +Y) }
	\leqslant C
	,
\]
so that
\[
	\frac Y{1 + Y}
	\leqslant
	\frac {y_0}{1 + y_0}
	e^{Ct}
	.
\]
Thus taking
\[
	T = \frac 1C \log
	\left(
		1 + \frac 1{ 1 + 2y_0}
	\right)
\]
we obtain $Y(t) \leqslant 2y_0$.
Finally, note that
\(
	y_0
	\leqslant
	\frac 32 \norm{ u(0) }_{\mathcal H}^2
	\leqslant
	\frac 3{ 8 C_{\mathcal H}^2 }
\)
and so
%
\[
	T(y_0) \geqslant
	\frac 1C \log
	\left(
		1 + \frac 1{ 1 + \frac 3{ 4 C_{\mathcal H}^2 }}
	\right)
\]
that can be taken as $T_1$.
%
%
%\end{proof}

\end{proof}

\end{comment}

%%%%%%%%%%%%%%%%%%%%%%%%%%%%%%%%%%%%%%%%%%%%%%%%%%%%%%%%%%%%%%%%%%%%%%%%%%%%%
\section{Proof of the main result}
%%%%%%%%%%%%%%%%%%%%%%%%%%%%%%%%%%%%%%%%%%%%%%%%%%%%%%%%%%%%%%%%%%%%%%%%%%%%%
\setcounter{equation}{0}
%%%%%%%%%%%%%%%%%%%%%%%%%%%%%%%%%%%%%%%%%%%%%%%%%%%%%%%%%%%%%%%%%%%%%%%%%%%%%

We construct a solution $u$ of \eqref{general_Duhamel}
iteratively on the intervals
\(
	[0, T_1]
	,
	[T_1, 2T_1]
\)
and so on.
Here the interval size $T_1$ is defined by Lemma \ref{energy_estimate_lemma}.
Staying under the assumptions of Theorem \ref{main_theorem},
we denote by $u_m$ solutions of Equation \eqref{R_Duhamel}
with $R = m \in \mathbb N$
given by Proposition \ref{R_existence_proposition},
where we subsequently set $t_0 = 0, T_1, 2T_1, \ldots$.
We define the stopping times
\begin{equation}
\label{stopping_time}
	\tau_m = \tau_m^{t_0} = \inf
	\left \{
		t \in [t_0, T_0]
		:
		\norm{ u_m(t) }_{H^{\sigma}} > m
	\right \}
\end{equation}
with the agreement $\inf \emptyset = T_0$.
Starting with $t_0 = 0$ we firstly show the following result.

\begin{lemma}
\label{consistency_lemma}
	For a.e. $\omega \in \Omega$, any
	\(
		m \in \mathbb N
	\)
	and each $t \in [0, \tau]$ with
	\(
		\tau(\omega) = \min \{ \tau_m(\omega), \tau_{m+1}(\omega) \}
		,
	\)
%	\(
%		t \in [ 0, \min \{ \tau_m(\omega), \tau_{m+1}(\omega) \} ]
%		,
%	\)
	it holds true that
	\(
		u_m(t) = u_{m+1}(t)
		.
	\)
\end{lemma}
\begin{proof}
We define
\[
	\widetilde u_i(t)
	=
	\left \{
		\begin{aligned}
			&
			u_i(t)
			& \mbox{if } &
			t \in [0, \tau]
			\\
			&
			\mathcal S(t, \tau) u_i(\tau)
			& \mbox{if } &
			t \in [\tau, T_0]
		\end{aligned}
	\right.
	,
	\quad
	i = m, m+1
	.
\]
At first we will show that
\(
	\widetilde u_m
\)
and
\(
	\widetilde u_{m+1}
\)
coincide in $X_T$ provided $T$ is sufficiently small.
Then we will finish the proof by an iteration procedure.
The difference of these functions has the form
\begin{multline*}
	\widetilde u_{m+1}(t) - \widetilde u_m(t)
	=
	\mathcal S(t, 0)
	\int_0^{ t \wedge \tau } \mathcal S(0, s)
	\left(
		f(\widetilde u_{m+1}(s)) - f(\widetilde u_m(s))
	\right)
	ds
	\\
	+
	\mathcal S(t, 0)
	\sum_j \gamma_j
	\int_0^{ t \wedge \tau } \mathcal S(0, s)
	\left(
		g(\widetilde u_{m+1}(s)) - g(\widetilde u_m(s))
	\right)
	dW_j(s)
	,
\end{multline*}
where the stochastic integral is estimated via
\begin{multline*}
	\mathbb E \sup_{ 0 \leqslant t \leqslant T }
	\norm
	{
		S_W(t, 0)
		\sum_j \gamma_j
		\int_0^t
		S(t - s) \chi_{ \{ s \leqslant \tau \} }(s) S_W(0, s)
		\left(
			g(\widetilde u_{m+1}(s)) - g(\widetilde u_m(s))
		\right)
		dW_j(s)
	}
	_{ H^{\sigma} }^2
	\\
	\leqslant
	C \mathbb E \int_0^T \chi_{ \{ s \leqslant \tau \} }(s)
	\norm
	{
		S_W(0, s)
		\left(
			g(\widetilde u_{m+1}(s)) - g(\widetilde u_m(s))
		\right)
	}
	_{ H^{\sigma} }^2
	ds
	\\
	\leqslant
	C \mathbb E \int_0^T \chi_{ \{ s \leqslant \tau \} }(s)
	\left(
		\norm
		{
			\widetilde u_{m+1}(s)
		}
		+	
		\norm
		{
			\widetilde u_m(s)
		}
		_{ H^{\sigma} }
	\right) ^2
	\norm
	{
		\widetilde u_{m+1}(s) - \widetilde u_m(s)
	}
	_{ H^{\sigma} }^2
	ds
	\\
	\leqslant
	C ( 2m + 1 )^2 T \mathbb E \sup_{[0, T]}
	\norm
	{
		\widetilde u_{m+1} - \widetilde u_m
	}
	_{ H^{\sigma} }^2
\end{multline*}
with the help of the Burkholder inequality for convolution
with the unitary group $S$,
see \cite[Lemma 3.3]{Gawarecki_Mandrekar}.
The first integral is estimated more straightforwardly,
notice a similar argument employed to $I$
in the proof of Proposition \ref{R_existence_proposition},
and so one obtains
\[
	\norm
	{
		\widetilde u_{m+1} - \widetilde u_m
	}
	_{ X_T }
	\leqslant
	C(m) \sqrt{T}
	\norm
	{
		\widetilde u_{m+1} - \widetilde u_m
	}
	_{ X_T }
	.
\]
Hence
\(
	\widetilde u_{m+1} = \widetilde u_m
\)
on $[0, T]$ for a.e. $\omega \in \Omega$
provided $T$ is chosen sufficiently small depending only on $m$.
Thus we can iterate this procedure to show that
\(
	\widetilde u_{m+1} = \widetilde u_m
\)
on the whole interval $[0, T_0]$,
which concludes the proof of the lemma.
\end{proof}

Our goal is to bound
\(
	\norm
	{
		u_m
	}
	_{ L^2 C( 0, T_1; H^{\sigma} ) }
\)
by a constant independent of $m \in \mathbb N$,
and so we will need to estimate
\(
	\norm{f(u_m)}_{H^{\sigma}}
	,
\)
\(
	\norm{g(u_m)}_{H^{\sigma}}
	,
\)
in particular.
This can be easily done with the help of
\[
	\norm{ \phi \psi }_{H^{\sigma}}
	\leqslant
	C( \sigma, \sigma_0 )
	\left(
		\norm{ \phi }_{H^{\sigma}}
		\norm{ \psi }_{H^{\sigma_0}}
		+
		\norm{ \phi }_{H^{\sigma_0}}
		\norm{ \psi }_{H^{\sigma}}
	\right)
\]
being true for any $\sigma \geqslant 0$ and $\sigma_0 > 1/2$,
see for example \cite[Estimate (3.12)]{Linares_Ponce}.

For a.e. $\omega \in \Omega$ and any $m \in \mathbb N$, $t \in [0, T_0]$
we have
\[
	\norm{u_m(t)}_{H^{\sigma}}
	\leqslant
	\norm{u_0}_{H^{\sigma}}
	+
	\int_0^t \norm{f(u_m(s))}_{H^{\sigma}} ds
	+
	\norm
	{
		\sum_j \gamma_j \int_0^t \mathcal S(0, s) g_m(u_m(s)) dW_j(s)
	}
	_{H^{\sigma}}
	,
\]
where
\(
	\norm{f(u_m(s))}_{H^{\sigma}}
	\leqslant
	C
	\left(
		\norm{u_m(s)}_{H^{\sigma_0}}
		+
		\norm{u_m(s)}_{H^{\sigma_0}}^2
	\right)
	\norm{u_m(s)}_{H^{\sigma}}
	.
\)
Now taking into account that
\(
	\norm
	{
		\mathcal S(0, s) g_m(u_m(s))
	}
	_{H^{\sigma}}
	\leqslant
	C
	\norm{u_m(s)}_{H^{\sigma_0}}
	\norm{u_m(s)}_{H^{\sigma}}
	,
\)
the stochastic integral can be estimated
by the  Burkholder inequality,
and so we obtain for any $0 < T \leqslant T_0$ the following inequality
\begin{equation}
\label{a_priori_estimate}
	\mathbb E \sup_{ t \in [0, T] } \norm{u_m(t)}_{H^{\sigma}}^2
	\leqslant	
	3 \mathbb E \norm{u_0}_{H^{\sigma}}^2
	+
	C \mathbb E \int_0^T
	\left(
		\norm{u_m(t)}_{H^{\sigma_0}}^2
		+
		\norm{u_m(t)}_{H^{\sigma_0}}^4
	\right)
	\norm{u_m(t)}_{H^{\sigma}}^2 dt
	,
\end{equation}
where $C$ depends only on $\sigma_0$, $\sigma$, $T_0$, $\sum_j \gamma_j^2$.
This inequality we will use iteratively on the intervals $[0, T_0 \wedge kT_1 ]$,
$k \in \mathbb N$, with $T_1$ found in Lemma \ref{energy_estimate_lemma}.
Let
\(
	\norm{u_0}_{ \mathcal H }
	\leqslant
	( 5 C_{ \mathcal H } )^{-1}
\)
a.e. on $\Omega$.
Consider the following stopping time
\[
	T_2^m = \inf
	\left \{
		t \in [0, T_0]
		:
		\norm{ u_m(t) }_{ \mathcal H } > ( 2 C_{ \mathcal H } )^{-1}
	\right \}
	.
\]
Then a.e. $T_1 \leqslant T_2^m$.
Indeed, assuming the contrary $T_1 > T_2^m$
one can deduce from \eqref{energy_norm_control}
and Lemma~\ref{energy_estimate_lemma} that
\[
	\norm{ u_m(T_2^m) }_{ \mathcal H }
	\leqslant
	\sqrt{ 2 \mathcal H( u_m(T_2^m) ) }
	\leqslant
	2 \sqrt{ \mathcal H( u_0 ) }
	\leqslant
	2 \sqrt{ 1 + C_{\mathcal H} \norm{ u_0 }_{ \mathcal H } }
	\norm{ u_0 }_{ \mathcal H }
	\leqslant
	\sqrt{ \frac{24}{125} } C_{ \mathcal H }^{-1}
	< ( 2 C_{ \mathcal H } )^{-1}
	,
\]
which contradicts to the definition of the stopping time $T_2^m$
due to continuity of $\norm{ u_m }_{ \mathcal H }$.
As a result $\norm{ u_m }_{ \mathcal H }$ stays bounded by $( 2 C_{ \mathcal H } )^{-1}$
on the interval $[0, T_1]$ for a.e. $\omega$,
and this simplifies \eqref{a_priori_estimate} in the following way
\[
	\mathbb E \sup_{ t \in [0, T] } \norm{u_m(t)}_{H^{\sigma}}^2
	\leqslant	
	3 \mathbb E \norm{u_0}_{H^{\sigma}}^2
	+
	C \int_0^T \mathbb E
	\sup_{s \in[0, t]}
	\norm{u_m(s)}_{H^{\sigma}}^2 dt
\]
holding true for any $0 < T \leqslant T_1$.
Hence by Gr\"onwall's lemma we obtain
\[
	\norm
	{
		u_m
	}
	_{ L^2 C( 0, T_1; H^{\sigma} ) }^2
	\leqslant
	3 \norm{u_0}_{L^2 H^{\sigma}}^2 e^{CT_1}
	= M
	,
\]
where $M$ does not depend on $m \in \mathbb N$.
Hence
\[
	\mathbb P ( \tau_m \geqslant T_1 )
	=
	\mathbb P
	\left(
		\norm
		{
			u_m
		}
		_{ C( 0, T_1; H^{\sigma} ) }
		\leqslant m
	\right)
	\geqslant
	1 - \frac 1{m^2} \mathbb E
	\norm
	{
		u_m
	}
	_{ C( 0, T_1; H^{\sigma} ) }^2
	\geqslant
	1 - \frac M{m^2}
	,
\]
and so
\(
	[0, T_1] \subset \cup_{m \in \mathbb N} [0, \tau_m(\omega)]
\)
for a.e. $\omega \in \Omega$.
Thus we can define $u$ on $[0, T_1]$ by assigning $u = u_m$ on $[0, \tau_m]$.
This is obviously a solution of \eqref{general_Duhamel} on $[0, T_1]$
satisfying $d \mathcal H(u) = 0$ and
\(
	\norm{ u }_{ \mathcal H } < ( 2 C_{ \mathcal H } )^{-1}
\)
for a.e. $\omega \in \Omega$.

Now one can repeat the argument on $[T_1, 2T_1]$ by constructing new solutions
$u_m$ of Equation \eqref{R_Duhamel} with the initial data $u(T_1)$ given
at the time moment $t_0 = T_1$.
The stopping times $\tau_m$ are defined by \eqref{stopping_time} with $t_0 = T_1$.
The fact that $\norm{ u_m }_{ \mathcal H }$
does not exceed the level $( 2 C_{ \mathcal H } )^{-1}$,
is guaranteed by the energy conservation, namely by
\(
	\mathcal H(u(T_1)) = \mathcal H(u_0)
\)
in the same manner as above.
The rest is similar,
and so we get a solution on $[T_1, 2T_1]$
with the constant energy equalled $\mathcal H(u_0)$.
After several repetitions of the argument we construct a solution on $[0, T_0]$.

It remains to prove the uniqueness.
Let
\(
	u_1, u_2 \in L^2( \Omega; C(0, T_0; H^{\sigma}(\mathbb R)) )
\)
solve Equation \eqref{general_Duhamel}.
For $R > 0$ we introduce
\[
	\tau_R = \inf
	\left \{
		t \in [0, T_0]
		:
		\max_{ i = 1,2 } \norm{ u_i(t) }_{H^{\sigma}} > R
	\right \}
	.
\]
Clearly,
for a.e. $\omega \in \Omega$ both $u_1$ and $u_2$
are solutions of \eqref{R_Duhamel} on $[0, \tau_R]$.
By Proposition \ref{R_existence_proposition}
it holds true that $u_1 = u_2$ on $[0, \tau_R]$
for a.e. $\omega \in \Omega$.
Taking $R \in \mathbb N$ and exploiting the time-continuity
of $u_1$, $u_2$ one obtains
\(
	u_1 = u_2
\)
on
\(
	[0, \lim_{R \to \infty} \tau_R]
\)
for a.e. $\omega \in \Omega$.
Now from sub-additivity and Chebyshev's inequality
we deduce
\[
	\mathbb P ( \tau_R \geqslant T_0 )
	=
	\mathbb P
	\left(
		\max_{ i = 1,2 } \norm{ u_i }_{C(0, T_0; H^{\sigma})} \leqslant R
	\right)
	\geqslant
	1 - \frac 1{R^2} \mathbb E
	\left(
		\norm{ u_1 }_{C(0, T_0; H^{\sigma})}^2
		+
		\norm{ u_2 }_{C(0, T_0; H^{\sigma})}^2
	\right)
	\to 1
\]
as $R \to \infty$, proving $u_1 = u_2$ on $[0, T_0]$.
This concludes the proof of Theorem \ref{main_theorem}.

%%%%%%%%%%%%%%%%%%%%%%%%%%%%%%%%%%%%%%%%%%%%%%%%%%%%%%%%%%%%%%%%%%%%%%%%%%%%%
%%%%%%%%%%%%%%%%%%%%%%%%%%%%%%%%%%%%%%%%%%%%%%%%%%%%%%%%%%%%%%%%%%%%%%%%%%%%%
%\begin{comment}
%%%%%%%%%%%%%%%%%%%%%%%%%%%%%%%%%%%%%%%%%%%%%%%%%%%%%%%%%%%%%%%%%%%%%%%%%%%%%
%%%%%%%%%%%%%%%%%%%%%%%%%%%%%%%%%%%%%%%%%%%%%%%%%%%%%%%%%%%%%%%%%%%%%%%%%%%%%

%%%%%%%%%%%%%%%%%%%%%%%%%%%%%%%%%%%%%%%%%%%%%%%%%%%%%%%%%%%%%%%%%%%%%%%%%%%%%
%%%%%%%%%%%%%%%%%%%%%%%%%%%%%%%%%%%%%%%%%%%%%%%%%%%%%%%%%%%%%%%%%%%%%%%%%%%%%
%\end{comment}
%%%%%%%%%%%%%%%%%%%%%%%%%%%%%%%%%%%%%%%%%%%%%%%%%%%%%%%%%%%%%%%%%%%%%%%%%%%%%
%%%%%%%%%%%%%%%%%%%%%%%%%%%%%%%%%%%%%%%%%%%%%%%%%%%%%%%%%%%%%%%%%%%%%%%%%%%%%

\vskip 0.05in
\noindent
{\bf Acknowledgments.}
{
	The author is grateful to the members of STUOD team for
	fruitful discussions and numerous helpful comments.
	The author acknowledges the support of the ERC EU project 856408-STUOD.
}

%%%%%%%%%%%%%%%%%%%%%%%%%%%%%%%%%%%%%%%%%%%%%%%%%%%%%%%%%%%%%%%%%%%%%%%%%%%%%
%%%%%%%%%%%%%%%%%%%%%%%%%%%%%%%%%%%%%%%%%%%%%%%%%%%%%%%%%%%%%%%%%%%%%%%%%%%%%
%%%%%%%%%%%%%%%%%%%%%%%%%%%%%%%%%%%%%%%%%%%%%%%%%%%%%%%%%%%%%%%%%%%%%%%%%%%%%
\bibliographystyle{acm}
\bibliography{bibliography}
%%%%%%%%%%%%%%%%%%%%%%%%%%%%%%%%%%%%%%%%%%%%%%%%%%%%%%%%%%%%%%%%%%%%%%%%%%%%%
%%%%%%%%%%%%%%%%%%%%%%%%%%%%%%%%%%%%%%%%%%%%%%%%%%%%%%%%%%%%%%%%%%%%%%%%%%%%%
%%%%%%%%%%%%%%%%%%%%%%%%%%%%%%%%%%%%%%%%%%%%%%%%%%%%%%%%%%%%%%%%%%%%%%%%%%%%%

\end{document}